\documentclass[a4paper,12pt]{amsart}
\usepackage{amssymb,amsfonts,pstricks,amsmath,fullpage}
\usepackage[english]{babel}
\usepackage{color}
\usepackage{xypic}

\newcommand{\ZZ}{\mathbf{Z}}
\newcommand{\PP}{\mathbf{P}}
\newcommand{\TT}{\mathbf{T}}
\newcommand{\RR}{\mathbf{R}}
\newcommand{\FF}{\mathbf{F}}

\newcommand{\conv}{\textup{conv}}

\newtheorem{theorem}{Theorem}
\newtheorem{lemma}[theorem]{Lemma}
\newtheorem{corollary}[theorem]{Corollary}

\newtheorem{conjecture}[theorem]{Conjecture}

\theoremstyle{definition}

\newtheorem{example}[theorem]{Example}

\theoremstyle{remark}
\newtheorem{remark}[theorem]{Remark}

\begin{document}

\title{Curves in characteristic $2$ with non-trivial
$2$-torsion}
\author{Wouter Castryck, Marco Streng, Damiano Testa}

\maketitle
\begin{abstract}
  \noindent Cais, Ellenberg and Zureick-Brown recently observed that over finite fields of characteristic two,
  all sufficiently general smooth plane projective curves of a given odd degree admit a non-trivial rational $2$-torsion point
  on their Jacobian. We extend their observation to curves given
  by Laurent polynomials with a fixed Newton polygon, provided that the polygon
  satisfies a certain combinatorial property. We also show that in each of these cases, the sufficiently
  general condition is implied by being ordinary.
  Our treatment includes many classical families, such
  as hyperelliptic curves of odd genus and $C_{a,b}$ curves. In the hyperelliptic case, we provide alternative
  proofs using an explicit description of the $2$-torsion subgroup.
\end{abstract}

\section{Introduction} \label{introduction}

The starting point of this article is a recent theorem by Cais, Ellenberg and Zureick-Brown~\cite[Thm.~4.2]{CEZB}, asserting
that over a finite field $k$ of characteristic $2$, almost all smooth plane projective curves of a given odd degree $d \geq 3$ have a non-trivial $k$-rational $2$-torsion point on their Jacobian. Here, `almost all' means that the corresponding proportion converges to $1$ as $\#k$ and/or $d$ tend to infinity.
The underlying observation is that such curves admit
\begin{itemize}
  \item a `geometric' $k$-rational half-canonical divisor $\Theta_\text{geom}$: the canonical class of a smooth plane projective curve of degree $d$ equals $(d-3) H$, where $H$ is the class of hyperplane sections; if $d$ is odd then $\frac{1}{2}(d-3) H$ is half-canonical,
  \item an `arithmetic' $k$-rational half-canonical divisor $\Theta_\text{arith}$ (whose class is sometimes called the \emph{can\-onical theta characteristic}), related to the fact that over a perfect field of characteristic $2$, the derivative of a Laurent series is always a square~\cite[p.~191]{mumford}.
\end{itemize}
The difference $\Theta_\text{geom} - \Theta_\text{arith}$ maps to a $k$-rational $2$-torsion point on the Jacobian. The proof of~\cite[Thm.~4.2]{CEZB} then amounts to showing that, quite remarkably, this point is almost always non-trivial.

There exist many classical families of curves admitting such a geometric half-canonical divisor. Examples include hyperelliptic curves of odd genus~$g$, whose canonical class is given by $(g-1) g^1_2$ (where $g^1_2$ denotes the hyperelliptic pencil), and smooth projective curves in $\PP_k^1 \times \PP_k^1$ of even bidegree $(a,b)$ (both $a$ and $b$ even, that is), where the canonical class reads $(a-2) R_1 + (b-2) R_2$ (here $R_1,R_2$ are the two rulings of $\PP_k^1 \times \PP_k^1$). The families mentioned so far are parameterized by sufficiently generic polynomials that are supported on the polygons
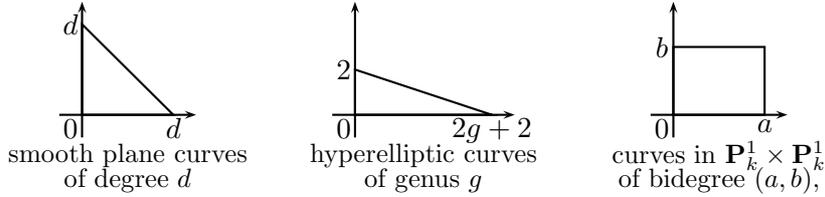
\begin{center}
\psset{unit=0.3cm}
\begin{pspicture}(-1,-3.2)(5,5)
\pspolygon[fillstyle=solid,linecolor=black](0,0)(4,0)(0,4)
\psline{->}(-1,0)(5,0) \psline{->}(0,-1)(0,5)
\rput(4,-0.6){\small $d$} \rput(-0.5,4){\small $d$}
\rput(-0.5,-0.6){\small $0$}
\rput(2,-1.8){\footnotesize smooth plane curves}
\rput(2,-2.9){\footnotesize of degree $d$}
\end{pspicture}
\qquad \qquad
\begin{pspicture}(-1,-3.2)(7,5)
\pspolygon[fillstyle=solid,linecolor=black](0,0)(6,0)(0,2)
\psline{->}(-1,0)(7,0) \psline{->}(0,-1)(0,5)
\rput(6,-0.7){\small $2g+2$} \rput(-0.5,2){\small $2$}
\rput(-0.5,-0.6){\small $0$}
\rput(3,-1.8){\footnotesize hyperelliptic curves}
\rput(3,-2.9){\footnotesize of genus $g$}
\end{pspicture}
\qquad \qquad
\begin{pspicture}(-1,-3.2)(5,5)
\pspolygon[fillstyle=solid,linecolor=black](0,0)(4,0)(4,3)(0,3)
\psline{->}(-1,0)(5,0) \psline{->}(0,-1)(0,5)
\rput(4,-0.5){\small $a$} \rput(-0.5,3){\small $b$}
\rput(-0.5,-0.6){\small $0$}
\rput(2,-1.8){\footnotesize curves in $\PP_k^1 \times \PP_k^1$}
\rput(2,-2.9){\footnotesize of bidegree $(a,b)$,}
\end{pspicture}
\end{center}
respectively. The following lemma, which is an easy consequence of the theory of toric surfaces (see Section~\ref{section_toric}), gives a purely combinatorial reason for the existence of a half-canonical divisor in these cases.

\begin{lemma} \label{combinatoriclemma}
Let $k$ be a perfect field and let $\Delta$ be a two-dimensional lattice polygon. For each edge $\tau \subset \Delta$, let $a_\tau X + b_\tau Y = c_\tau$ be its supporting line,
where $\gcd(a_\tau,b_\tau) = 1$. Suppose that the  system of congruences
\begin{equation} \label{congs}
\left\{ \ a_\tau X + b_\tau Y \equiv c_\tau + 1 \pmod{2}  \ \right\}_{\tau \text{ edge of } \Delta}
\end{equation}
admits a solution in $\ZZ^2$. Then any sufficiently general Laurent polynomial $f \in k[x^{\pm 1}, y^{\pm 1}]$ that is supported on $\Delta$ defines a curve carrying a $k$-rational half-canonical divisor
on its non-singular complete model.
\end{lemma}

\noindent In the proof of Lemma~\ref{combinatoriclemma} below, where we describe this half-canonical divisor explicitly, we will be more precise on the meaning of `sufficiently general'.

Here again, when specializing to characteristic $2$, there is, in addition, an arithmetic $k$-rational half-canonical divisor. So it is natural to wonder whether the proof of~\cite[Thm.~4.2]{CEZB} still applies in these cases. We will show that it usually does.

\begin{theorem} \label{maintheorem}
Let $\Delta$ be a two-dimensional lattice polygon satisfying the conditions of Lemma~\ref{combinatoriclemma}, where in addition we assume that $\Delta$ is not unimodularly equivalent to
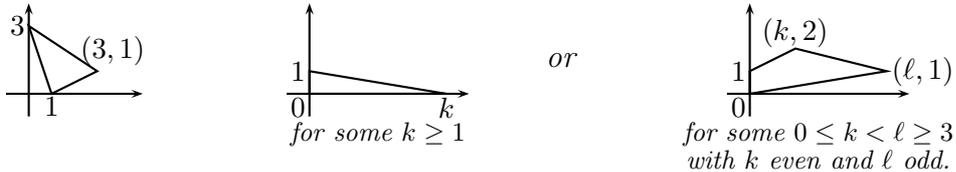
\begin{center}
\psset{unit=0.3cm}
\begin{pspicture}(-1,-3.2)(5,4)
\pspolygon[fillstyle=solid,linecolor=black](1,0)(0,3)(3,1)
\psline{->}(-1,0)(5,0) \psline{->}(0,-1)(0,4)
\rput(1,-0.55){\small $1$} \rput(-0.5,3){\small $3$}
\rput(3.7,1.9){\small $(3,1)$}
\end{pspicture}
\quad \quad
\begin{pspicture}(-4,-3.2)(12,4)
\pspolygon[fillstyle=solid,linecolor=black](0,0)(6,0)(0,1)
\psline{->}(-1,0)(7,0) \psline{->}(0,-1)(0,4)
\rput(6,-0.6){\small $k$} \rput(-0.5,1){\small $1$}
\rput(-0.5,-0.6){\small $0$}
\rput(3,-1.8){\footnotesize for some $k \geq 1$}
\rput(11,1.5){or}
\end{pspicture}
\quad \quad
\begin{pspicture}(-4,-3.2)(6,4)
\pspolygon[fillstyle=solid,linecolor=black](0,0)(6,1)(2,2)(0,1)
\psline{->}(-1,0)(7,0) \psline{->}(0,-1)(0,4)
\rput(-0.5,1){\small $1$}
\rput(-0.5,-0.6){\small $0$}
\rput(3,-1.8){\footnotesize for some $0 \leq k < \ell \geq 3 $}
\rput(3,-2.9){\footnotesize with $k$ even and $\ell$ odd.}
\rput(7.6,1){\small $(\ell,1)$}
\rput(2,2.7){\small $(k,2)$}
\end{pspicture}
\end{center}
Then there exists a non-empty Zariski open subset $S_\Delta / \FF_2$ of the space of
Laurent polynomials that are supported on~$\Delta$ having the following property.
For every perfect field $k$ of characteristic $2$ and every $f \in S_\Delta(k)$,
the Jacobian of the non-singular complete model of the
curve defined by $f$ has a non-trivial $k$-rational $2$-torsion point.
\end{theorem}

\noindent (Right before the proof of Theorem~\ref{maintheorem} we will define the set $S_\Delta$ explicitly.) As a consequence, if $k$
is a finite field of characteristic $2$, then the proportion of Laurent polynomials
that are supported on~$\Delta$, which define a curve whose Jacobian has a
non-trivial $k$-rational $2$-torsion point, tends to $1$ as $\#k \rightarrow \infty$.
See the end of
Section~\ref{section_mainresult}, where we also discuss asymptotics for
increasing dilations of $\Delta$, i.e.\ the analogue of $d \rightarrow \infty$ in the smooth plane curve case.
In Section~\ref{hassewitt} we give sufficient conditions that have a more
arithmetic flavor, involving the rank of the Hasse-Witt matrix.

These observations seem new even for hyperelliptic curves of odd genus\footnote{In
view of the asymptotic consequences discussed in Section~\ref{section_mainresult}, this observation shows that
\cite[Principle~3]{CFHS} can fail for $g > 2$.} (even though this is a well-known fact for the subfamily of hyperelliptic curves having a prescribed
$k$-rational Weierstrass point; see below).
In this case we can give alternative proofs using an explicit description of the $2$-torsion subgroup; see Section~\ref{hyperelliptic}.
Another interesting class of examples is given by the polygons
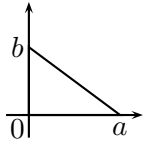
\begin{center}
\psset{unit=0.3cm}
\begin{pspicture}(-1,-1)(5,5)
\pspolygon[fillstyle=solid,linecolor=black](0,0)(4,0)(0,3)
\psline{->}(-1,0)(5,0) \psline{->}(0,-1)(0,5)
\rput(4,-0.6){\small $a$} \rput(-0.5,3){\small $b$}
\rput(-0.5,-0.6){\small $0$}
\end{pspicture}
\end{center}
where $a$ and $b$ are not both even. The case $a = b$ corresponds to the smooth plane curves of odd degree considered in~\cite{CEZB}. The case $\gcd(a,b) = 1$ corresponds to so-called $C_{a,b}$ curves. The
case $b=2$, $a = 2g+1$ (a subcase of the latter) corresponds
to hyperelliptic curves having a prescribed $k$-rational Weierstrass point $P$. Note that in this case $g^1_2 \sim 2P$, so there is indeed always a $k$-rational half-canonical divisor, regardless of the parity of~$g$.
\begin{remark}
This explains why Denef and Vercauteren had to allow a factor $2$ while generating
cryptographic hyperelliptic and $C_{a,b}$ curves in characteristic $2$; see Sections~6 of~\cite{DV1,DV2}.
\end{remark}
\noindent Finally, the case $b=3$, $a \geq 4$ corresponds to trigonal curves having maximal Maroni invariant (that is trigonal curves for which the series $(h^0(n g^1_3))_{n \in \ZZ_{\geq 0}}$ starts increasing by steps of $3$ as late as the Riemann-Roch theorem allows it to do); if $a=6$, these are exactly the genus-$4$ curves having a unique $g^1_3$.

We conclude by stressing
that the results in this paper are unlikely to generalize to characteristic $p > 2$, by lack of an appropriate analogue of our arithmetic half-canonical
divisor $\Theta_\text{arith}$.

\section{Half-canonical divisors from toric geometry} \label{section_toric}

Let $k$ be a perfect field, let $f = \sum_{(i,j) \in \ZZ^2} c_{i,j} x^i y^j \, \in k[x^{\pm 1}, y^{\pm 1}]$ be a
Laurent polynomial, and let
\[ \Delta(f) = \conv \left\{ \, \left. (i,j) \in \ZZ^2 \, \right| \, c_{i,j} \neq 0 \,    \right\} \]
be the Newton polygon of~$f$, which we assume to be two-dimensional. We say that $f$ is \emph{non-degenerate with respect to its Newton polygon} if for every face $\tau \subset \Delta(f)$ (vertex, edge, or $\Delta(f)$ itself)
the system
\[ f_\tau = \frac{\partial f_\tau}{\partial x} = \frac{\partial f_\tau}{\partial y} = 0 \qquad \text{with } f_\tau = \sum_{(i,j) \in \tau \cap \ZZ^2} c_{i,j} x^i y^j \]
has no solutions\footnote{Note that this is in fact automatically true if $\tau$ is a vertex.} over an algebraic closure of $k$.
For a given two-dimensional lattice polygon $\Delta$, we say that $f$ is \emph{$\Delta$-non-degenerate}
if $\Delta(f) = \Delta$ and $f$ is non-degenerate with respect to its Newton polygon.
The condition of $\Delta$-non-degeneracy is generically
satisfied, in the sense that it is characterized by the non-vanishing of
\[ \rho_\Delta := \text{Res}_\Delta \left(f, x \frac{\partial f}{ \partial x}, y \frac{\partial f}{\partial y} \right) \, \in \ZZ[  c_{i,j} | (i,j) \in \Delta \cap \ZZ^2] \]
(where $\text{Res}_\Delta$ is the sparse resultant; $\rho_\Delta$ does not vanish identically in any characteristic~\cite[\S2]{CastryckVoight}).
Non-degenerate Laurent polynomials are always (absolutely) irreducible.

To a two-dimensional lattice polygon $\Delta$ one can associate a toric surface $\text{Tor}_k(\Delta)$, which
is a compactification of $\TT_k^2 = \text{Spec} \, k[x^{\pm 1}, y^{\pm 1}]$ to which
the natural self-action of the latter extends algebraically.
This extended action decomposes $\text{Tor}_k(\Delta)$ in a finite number of orbits,
which naturally correspond (in a dimension-preserving manner) to the
faces of $\Delta$; for each face $\tau$, write $O(\tau)$ for
the according orbit.
Now if $f \in k[x^{\pm 1}, y^{\pm 1}]$
is a $\Delta$-non-degenerate Laurent polynomial,
the non-degeneracy condition with respect to $\Delta$ itself
ensures that it cuts out a non-singular
curve $C_f$ in $\TT_k^2 = O(\Delta)$. Similarly, one finds that its compactification
$C'_f$ in $\text{Tor}_k(\Delta)$ does not contain any of the
zero-dimensional $O(\tau)$'s, and that it intersects the one-dimensional $O(\tau)$'s transversally.
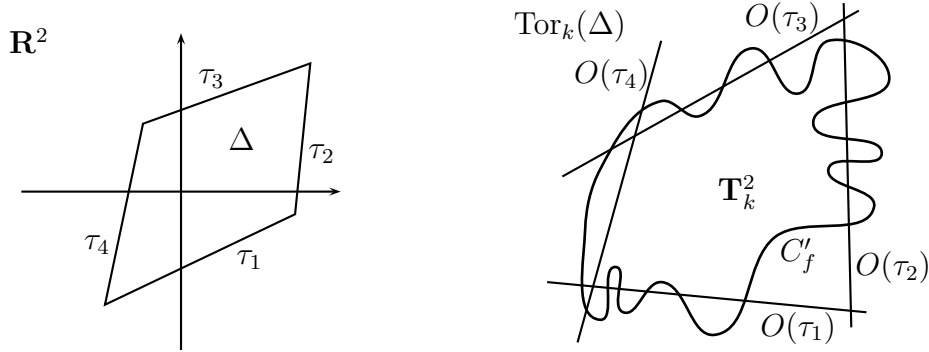
\begin{center}
\begin{pspicture}(-2.2,-2.2)(2.2,2.6)
\pspolygon[linecolor=black](-1,-1.5)(1.5,-0.3)(1.7,1.7)(-0.5,0.9)
\psline{->}(-2.1,0)(2.1,0) \psline{->}(0,-2.1)(0,2.1)
\rput(0.8,0.7){$\Delta$} \rput(0.9,-0.9){$\tau_1$}
\rput(1.85,0.5){$\tau_2$} \rput(0.4,1.5){$\tau_3$}
\rput(-1.1,-0.7){$\tau_4$} \rput(-2,2){$\mathbf{R}^2$}
\end{pspicture} \qquad \qquad \qquad
\begin{pspicture}(-2.5,-2.2)(2.1,2.5)
\psccurve[linewidth=1pt](-1.8,-1.7)(-1.6,-1)(-1.5,-1.6)(-1,-1.2)(-0.3,-1.9)(0.5,-0.6)(1.8,-0.2)%
                        (1.2,0.3)(1.9,0.5)(1,0.9)(2,1.5)(1.2,2)(0.8,1.3)(0.2,1.9)(-0.5,1)(-1,1.2)%
                        (-1.8,0.3)(-2,-0.5)
\psline{-}(-2.5,-1.2)(1.7,-1.6) \psline{-}(1.5,-1.8)(1.4,2.5)
\psline{-}(1.6,2.4)(-2.3,0.2) \psline{-}(-1,2)(-2.1,-2)
\rput(0,0){$\mathbf{T}^2_k$}
\rput(0.8,-0.85){$C_f'$} \rput(2.05,-1){$O(\tau_2)$}
\rput(-1.65,1.55){$O(\tau_4)$}
\rput(0.8,-1.8){$O(\tau_1)$}
\rput(0.6,2.3){$O(\tau_3)$}
\rput(-2.2,2.2){$\text{Tor}_k(\Delta)$}
\end{pspicture}
\end{center}
In particular, since $\text{Tor}_k(\Delta)$ is normal, the
non-degeneracy of $f$ implies that $C'_f$ is a non-singular complete model of $C_f$.
See \cite[\S3-4]{linearpencils} and~\cite[\S2]{CDV}
for more details.

\begin{example} \label{smoothplaneexample}
  Assume that $\Delta = \conv \{ (0,0), (d,0), (0,d) \}$.
  In this case $\text{Tor}_k(\Delta)$ is just the projective plane, and the toric orbits are
  \begin{itemize}
  \item $\TT_k^2 = O(\Delta)$,
  \item the three coordinate points $(1:0:0)$, $(0:1:0)$ and $(0:0:1)$,
  which are the orbits of the form $O(\text{vertex})$,
  \item the three coordinate axes from which the coordinate points are removed: these are the orbits of the form $O(\text{edge})$.
  \end{itemize}
Thus $C'_f$ is a non-singular projective plane curve that is non-tangent to any of the coordinate axes, and that does not contain
any of the coordinate points.
This is essentially an if-and-only-if: an absolutely irreducible Laurent polynomial $f \in k[x^{\pm 1}, y^{\pm 1}]$, for which $\Delta(f) \subset \Delta$,
is $\Delta$-non-degenerate if and only if its zero locus in $\mathbf{T}_k^2$ compactifies to a non-singular degree $d$ curve
in $\mathbf{P}^2_k$ that
is non-tangent to the coordinate axes, and that does not contain the coordinate points.
\end{example}

\begin{example}
  Let $g \geq 2$ be an integer, and consider $f = y^2 + h_1(x)y + h_0(x)$, where $\deg h_1 \leq g + 1$, $\deg h_0 = 2g + 2$, and $h_0(0) \neq 0$. Then
  $\Delta(f) = \conv \{ (0,0), (2g+2,0), (0,2) \}$, and $\text{Tor}_k(\Delta(f))$ is the weighted projective plane $\PP_k(1:g+1:1)$. Here again,
  if $f$ is non-degenerate with respect to its Newton polygon then $C'_f$ is a non-singular curve that is non-tangent to the coordinate axes
  and that does not contain any coordinate points. In this case $C'_f$ is a hyperelliptic
  curve of genus $g$ (cf.\ Remark~\ref{remarkgenus}).
\end{example}

Now for each edge $\tau \subset \Delta$ let
$\nu_\tau \in \ZZ^2$ be the inward pointing primitive normal vector to $\tau$, let $p_\tau$ be any element of $\tau \cap \ZZ^2$, and let $D_\tau$ be the $k$-rational divisor on $C'_f$ cut out by $O(\tau)$.
Using the $\Delta$-non-degeneracy of $f$ one can prove
\begin{equation} \label{divisorequality}
 \text{div} \frac{dx}{xy \frac{\partial f}{\partial y}} = \sum_{\tau \text{ edge}} (- \langle \nu_\tau , p_\tau \rangle - 1) D_\tau.
\end{equation}
Here $\langle \cdot, \cdot \rangle$ is the standard inner product on $\RR^2$.
See~\cite[Cor.~2.7]{CDV} for an elementary but elaborate proof of (\ref{divisorequality}).
It is possible to give a more conceptual proof using adjunction theory, along the lines of \cite[Prop.\,10.5.8]{coxlittleschenck}.

\begin{remark} \label{rem:sing}
From the theory of sparse resultants it follows that $\partial f / \partial y$ does
not vanish identically,
so that the left-hand side of (\ref{divisorequality}) makes sense.
Note also that $0 = df = \frac{\partial f}{\partial x} dx + \frac{\partial f}{\partial y} dy$, so
we could as well have written
\begin{equation} \label{divisorequalityalternative}
\text{div} \frac{dy}{xy \frac{\partial f}{\partial x}}.
\end{equation}
\end{remark}

\begin{proof}[Proof of Lemma~\ref{combinatoriclemma}]
Assume that $f$ is $\Delta$-non-degenerate (which, as mentioned above, is a non-empty Zariski open condition).
Let $(i_0,j_0) \in \ZZ^2$ be a solution to
the given system of congruences. We claim that the translated polygon $(-i_0, -j_0) + \Delta$ is such that all corresponding $\langle \nu_\tau , p_\tau\rangle$'s are odd. To see this, note that in this case $(0,0)$ is a solution to the according system of congruences (\ref{congs}). This implies that all $c_\tau$'s are odd.
Together with $\langle \nu_\tau , p_\tau\rangle = \pm c_\tau$ this yields the claim. 
So by applying the above to $x^{-i_0}y^{-j_0} f$, we find that
\[ \Theta_\text{geom} = \sum_{\tau \text{ edge}} \frac{- \langle \nu_\tau , p_\tau \rangle - 1}{2} D_\tau \]
is a $k$-rational half-canonical divisor on $C'_{x^{-i_0}y^{-j_0} f} = C'_f$.
\end{proof}

\noindent \textbf{Example~\ref{smoothplaneexample} (continued).} Assume that $d$ is odd, so that the conditions from Lemma~\ref{combinatoriclemma}
are satisfied. Applying the above proof with $(i_0,j_0) = (1,1)$ yields
\[ \Theta_\text{geom} = \frac{d-3}{2} D_\infty \]
where $D_\infty$ is the divisor cut out by the line at infinity. So 
we recover the divisor class mentioned in the introduction.

\begin{remark} \label{remarknondegversussmooth}
Still assume that $\Delta = \conv \{ (0,0), (d,0), (0,d) \}$ with $d$ odd. We already noted that the condition of non-degene\-racy restricts our attention to smooth plane curves of degree $d$ that do not contain the coordinate points and that intersect the coordinate axes transversally. But of course \emph{any} smooth plane curve of degree $d$ carries a $k$-rational half-canonical divisor. This shows that the non-degeneracy condition, even though it is generically satisfied,
is sometimes a bit stronger than needed.\footnote{The reader might want to note that there always exists an automorphism of $\PP_k^2$ that puts our smooth plane curve in a non-degenerate position (at least if $\#k$ is sufficiently large). But for
more general instances of $\Delta$, the automorphism group of $\text{Tor}_k(\Delta)$ may be much smaller (e.g.\ the only automorphisms may be the ones coming from the $\TT_k^2$-action), in which case it
might be impossible to resolve tangency to the one-dimensional toric orbits.}
For a general two-dimensional lattice polygon $\Delta$,
the according weaker condition reads that $f$ is \emph{$\Delta$-toric}, meaning
that $\Delta(f) \subset \Delta$, that $\Delta(f)^{(1)} = \Delta^{(1)}$, and
that $C_f$ compactifies to a non-singular curve $C_f'$ in $\text{Tor}_k(\Delta)$.
Here $\Delta^{(1)}$ denotes the lattice polygon obtained by taking
the convex hull of the $\mathbf{Z}^2$-points that lie in the interior of $\Delta$, and similarly for $\Delta(f)^{(1)}$.
See \cite[\S4]{linearpencils} for more background on this notion.
Now we have to revisit Remark~\ref{rem:sing}, however:
there do exist instances of $\Delta$-toric Laurent polynomials $f \in k[x^{\pm 1}, y^{\pm 1}]$ for which
$\partial f / \partial y$ \emph{does} vanish identically (example: take $f = 1 + x^2y^2 + x^3y^2$ and $\Delta = \Delta(f)$).
For these instances
the left-hand side of (\ref{divisorequality}) does not make sense. But in that
case $\partial f / \partial x$ does not vanish identically (otherwise $C_f$ would have singularities), and one can prove that (\ref{divisorequality}) holds with the left-hand side replaced by (\ref{divisorequalityalternative}).
\end{remark}

\begin{remark} \label{remarkgenus}
We mention two other well-known features of $\Delta$-non-degenerate (or $\Delta$-toric)
Laurent polynomials, that can be seen as consequences to~\eqref{divisorequality}; see for instance~\cite{linearpencils,CastryckVoight} and
the references therein:
\begin{itemize}
\item the genus of $C'_f$ equals $\# \left(\Delta^{(1)} \cap \ZZ^2 \right)$,
\item
if $\# \left(\Delta^{(1)} \cap \ZZ^2 \right) \geq 2$, then $C'_f$ is hyperelliptic if
and only if $\Delta^{(1)} \cap \ZZ^2$ is contained in a line.
\end{itemize}
\end{remark}

\section{Proof of the main result} \label{section_mainresult}

\begin{lemma} \label{solutionsinside}
Let $\Delta$ be a two-dimensional lattice polygon and
suppose as in Lemma~\ref{combinatoriclemma} that~\eqref{congs} admits a solution in $\ZZ^2$. If $\Delta$ is not among the polygons excluded in the hypothesis of Theorem~\ref{maintheorem}, then there is a solution of~\eqref{congs} contained in $\Delta \cap \ZZ^2$.
\end{lemma}

\begin{proof}
Let us first classify all two-dimensional lattice polygons $\Delta$ for which the reduction-modulo-$2$ map $\pi_\Delta \colon \Delta \cap \ZZ^2 \rightarrow \left( \ZZ/(2) \right)^2$ is not surjective.
If the interior lattice points of $\Delta$ lie on a line,
then surjectivity fails if and only if $\Delta$ is among
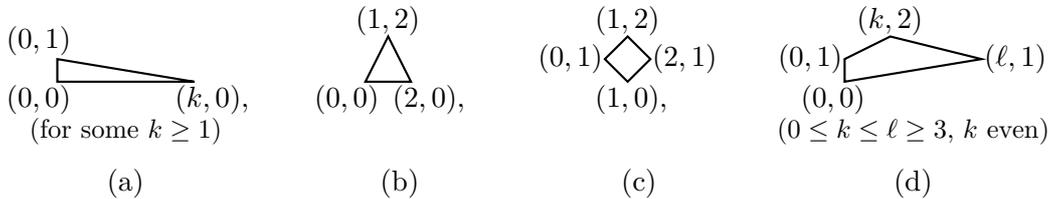
\begin{center}
\psset{unit=0.3cm}
\begin{pspicture}(-1,-4.5)(7,3.5)
\pspolygon[fillstyle=solid,linecolor=black](0,0)(6,0)(0,1)
\rput(-0.8,-0.8){\small $(0,0)$}
\rput(-0.8,1.8){\small $(0,1)$}
\rput(6.8,-0.8){\small $(k,0)$,}
\rput(3,-2.2){\footnotesize (for some $k \geq 1$)}
\rput(3,-4.5){\small (a)}
\end{pspicture}
\qquad \quad \ \
\begin{pspicture}(-1,-4.5)(4,3.5)
\pspolygon[fillstyle=solid,linecolor=black](0,0)(2,0)(1,2)
\rput(-0.8,-0.8){\small $(0,0)$}
\rput(1,2.7){\small $(1,2)$}
\rput(2.8,-0.8){\small $(2,0)$,}
\rput(1.5,-4.5){\small (b)}
\end{pspicture}
\qquad \quad \ \
\begin{pspicture}(-1,-4.5)(4,3.5)
\pspolygon[fillstyle=solid,linecolor=black](1,0)(2,1)(1,2)(0,1)
\rput(-1.5,1){\small $(0,1)$}
\rput(1,2.7){\small $(1,2)$}
\rput(3.5,1){\small $(2,1)$}
\rput(1,-0.8){\small \phantom{,}$(1,0)$,}
\rput(1.5,-4.5){\small (c)}
\end{pspicture}
\qquad \quad \ \
\begin{pspicture}(-1,-4.5)(7.7,3.5)
\pspolygon[fillstyle=solid,linecolor=black](0,0)(6,1)(2,2)(0,1)
\rput(-1.5,1){\small $(0,1)$}
\rput(2,2.7){\small $(k,2)$}
\rput(7.5,1){\small $(\ell,1)$}
\rput(-0.5,-0.8){\small $(0,0)$}
\rput(3,-2.2){\footnotesize ($0 \leq k \leq \ell \geq 3$, $k$ even)}
\rput(3,-4.5){\small (d)}
\end{pspicture}
\end{center}
(up to unimodular equivalence). This assertion follows from Koelman's classification; see~\cite[Ch.~4]{Koelman} or~\cite[Thm.~10]{movingout}.
Now any two-dimensional lattice polygon $\Delta$ can be peeled into `onion skins', by subsequently taking the convex hull of the interior lattice points, until
one ends up with a lattice polygon whose interior lattice points are contained in a line.
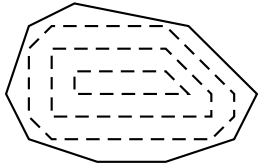
\begin{center}
\psset{unit=0.3cm}
\begin{pspicture}(0,0)(11,7)
\pspolygon[fillstyle=solid,linecolor=black](0,3)(1,1)(4,0)(7,0)(10,1)(11,3)(8,6)(3,7)(1,6)
\pspolygon[fillstyle=solid,linestyle=dashed,linecolor=black](1,2)(2,1)(9,1)(10,2)(10,3)(7,6)(2,6)(1,5)
\pspolygon[fillstyle=solid,linestyle=dashed,linecolor=black](2,2)(9,2)(9,3)(7,5)(2,5)
\pspolygon[fillstyle=solid,linestyle=dashed,linecolor=black](3,3)(8,3)(7,4)(3,4)
\end{pspicture}
\end{center}
If $\pi_\Delta$ is not surjective, then clearly $\pi_{\Omega}$ is not
surjective for each onion skin $\Omega$. In particular, the last onion skin must necessarily be among (a-d).

But for a lattice polygon to arise as an onion skin of a strictly larger
lattice polygon $\Delta$ is a stringent condition. Using the criterion from~\cite[Lem.~9-11]{HaaseSchicho} one sees that the only polygons among
(a-d) of this type are the polygons (a) with $k = 1$ or $k = 2$, the polygon (b) and the polygon (c). The same
criterion shows that the only instance of such a larger $\Delta$ for which $\pi_\Delta$ is not surjective is
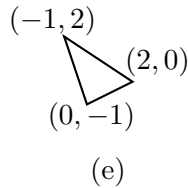
\begin{center}
\psset{unit=0.3cm}
\begin{pspicture}(-1,-3)(4,4.5)
\pspolygon[fillstyle=solid,linecolor=black](1,0)(0,3)(3,1)
\rput(1.2,-0.55){\small $(0,-1)$} \rput(-0.5,3.7){\small $(-1,2)$}
\rput(4.1,1.9){\small $(2,0)$}
\rput(1.5,-3){\phantom{.} \small (e)}
\end{pspicture}
\end{center}
(up to unimodular equivalence). The latter, again by~\cite[Lem.~9-11]{HaaseSchicho}, is not an onion skin of a strictly bigger lattice polygon itself.
This ends the classification: up to unimodular equivalence, the instances of $\Delta$ for which $\pi_\Delta$ is not surjective are
(a)-(e).

Now let $\Delta$ be a two-dimensional lattice polygon and suppose that~\eqref{congs} admits a solution in $\ZZ^2$. If
$\pi_\Delta$ is surjective, then it clearly also admits a solution in $\Delta \cap \ZZ^2$. So we may assume that
$\Delta$ is among (a-e). Then the lemma follows by noting that cases (b), (c) and (d) with $\ell$ even admit the solution $(1,1) \in \Delta \cap \ZZ^2$, and that cases (a), (e) and (d) with $\ell$ odd were excluded in the \'enonc\'e.
\end{proof}

\begin{remark}
Because of Remark~\ref{remarkgenus}, the excluded polygons correspond to certain classes of smooth plane quartics, rational curves, and hyperelliptic curves, respectively.
\end{remark}

\noindent We can now define the variety $S_\Delta$ mentioned in the statement of Theorem~\ref{maintheorem}. Namely, we will prove the existence of a
non-trivial $k$-rational $2$-torsion point under the assumption that
\begin{itemize}
  \item $f$ is $\Delta$-non-degenerate (i.e.\ the genericity assumption from Lemma~\ref{combinatoriclemma}), and
  \item for at least one solution $(i_0,j_0) \in \Delta \cap \ZZ^2$ to the system of congruences~\eqref{congs}, the corresponding coefficient $c_{i_0,j_0}$ is non-zero.
\end{itemize}
So we can let $S_\Delta$ be defined by $c_{i_0,j_0} \rho_\Delta \neq 0$.

\begin{remark} \label{hereagain}
Here again, one can weaken the
condition of being $\Delta$-non-degenerate to being $\Delta$-toric, as
described in Remark~\ref{remarknondegversussmooth}. When that stronger version is applied
to $\Delta = \conv \{ (0,0), (d,0), (0,d) \}$ with $d$ odd, one exactly recovers~\cite[Thm.~4.2]{CEZB}.
\end{remark}

\begin{proof}[Proof of Theorem~\ref{maintheorem}]
By replacing $f$ with $x^{-i_0}y^{-j_0}f$
if needed, we assume that $(0,0) \in \Delta$ is a solution to the system of congruences~\eqref{congs} and that the constant term of $f$ is non-zero.
As explained in~\cite[p.~191]{mumford}, $C_f'$ comes equipped with a $k$-rational divisor $\Theta_\text{arith}$ such that $2 \Theta_\text{arith} = \text{div} \, dx$.
(Recall that the derivative of a Laurent series over $k$ is always a square, so the order of $dx$ at a point of $C_f'$ is indeed even.)
On the other hand, Lemma~\ref{combinatoriclemma} and its proof provide us with a $k$-rational divisor $\Theta_\text{geom}$ such that
\[
2\Theta_\text{geom} = \text{div} \, \frac{dx}{xy \frac{\partial f}{\partial y}}.
\]
In order to prove that $\Theta_\text{geom} \not \sim \Theta_\text{arith}$ (and hence that $\text{Jac}(C'_f)$ has a non-trivial $k$-rational $2$-torsion point), we
need to show that
\[
xy \frac{\partial f}{\partial y}
\]
is a non-square when considered as an element of the function field $k(C_f)$. If it were a square, then there would exist Laurent polynomials $\alpha, G, H$ such that
\begin{equation} \label{eq1}
H^2 xy \frac{\partial f}{\partial y} + \alpha f =  G^2 \qquad \text{in $k[x^{\pm 1}, y^{\pm 1}]$},
\end{equation}
where $f \nmid H$.
Taking derivatives with respect to $y$ yields
\[ (\alpha + H^2 x) \frac{\partial f}{\partial y} = \frac{\partial \alpha}{\partial y} f,\]
which together with~\eqref{eq1} results in
\[ \left( (\alpha + H^2 x) \alpha + H^2 xy \frac{\partial \alpha}{\partial y} \right) f = (\alpha + H^2 x) G^2.\]
Since $f$ is irreducible, it follows that $f \mid (\alpha + H^2 x)$ or $f \mid G^2$. Using~\eqref{eq1} and $f \nmid H$, the latter implies that $f  \mid \frac{\partial f}{\partial y}$, which is
a contradiction (by the theory of sparse resultants, see Remark~\ref{rem:sing}; one can alternatively repeat the argument using (\ref{divisorequalityalternative}) if wanted).
So we know that $f \mid (\alpha + H^2 x)$. Along with~\eqref{eq1} we conclude that there exists a Laurent polynomial $\beta \in k[x^{\pm 1}, y^{\pm 1}]$ such that
\[ H^2 x \left( y \frac{\partial f}{\partial y} + f \right)  + \beta f^2 = G^2.\]
Taking derivatives with respect to $x$ yields
\[ H^2 \left( f + x \frac{\partial f}{\partial x} + y \frac{\partial f}{\partial y} + xy \frac{\partial^2 f}{\partial x \partial y} \right) + \frac{\partial \beta}{\partial x} f^2 = 0.\]
Since $f$ has a non-zero constant term, the large factor between brackets is non-zero. On the other hand, since $f \nmid H$, it must be a multiple of $f^2$.
Note that $\Delta(f^2) = 2\Delta(f)$, while $\Delta(f+\cdots + xy\partial^2 f / (\partial x\partial y)) \subset \Delta(f)$. This is a contradiction.
\end{proof}
\vspace{0.3cm}
\noindent We end this section by discussing some asymptotic consequences to Theorem~\ref{maintheorem}.

\subsubsection*{Growing field size.}
  Let $\Delta$ be a two-dimensional
  lattice polygon satisfying the conditions of Theorem~\ref{maintheorem}.
  Let $k$ be a finite field of characteristic $2$. Because
  non-degeneracy is characterized by the non-vanishing of $\rho_\Delta$, the
  proportion of $\Delta$-non-degenerate Laurent polynomials $f \in k[x^{\pm 1}, y^{\pm 1}]$
  (amongst all
  Laurent polynomials that are supported on $\Delta$) converges to $1$
  as $\# k \rightarrow \infty$. Then Theorem~\ref{maintheorem} implies:
  \[ \lim_{\#k \rightarrow \infty} \text{Prob} \left( \, \text{Jac}(C_f')(k)[2] \neq 0 \, \left| \, \text{$f \in k[x^{\pm 1}, y^{\pm 1}]$ is
  $\Delta$-non-degenerate} \, \right. \right) = 1. \]
   As soon as $\# (\Delta^{(1)} \cap \ZZ^2) \geq 2$ this is deviating statistical behavior: in view
   of Katz-Sarnak-Chebotarev-type density theorems~\cite[Theorem~9.7.13]{KatzSarnak}, for a general smooth proper family of
  genus $g$ curves, one expects that the
  probability of having a non-trivial
rational $2$-torsion point on the Jacobian
  approaches the chance that a random matrix in $\text{GL}_g(\FF_2)$ satisfies
  $\det(M - \text{Id}) = 0$, which is
  \[ - \sum_{r = 1}^g \prod_{j=1}^r \frac{1}{1- 2^j}\]
  by~\cite[Thm.~6]{CFHS}. For $g = 1, 2, 3, 4, \dots$, these probabilities are $1, \frac{2}{3}, \frac{5}{7}, \frac{32}{45}, \dots$ (converging to about $0.71121$).

  In the table below we denote by $\square_i$ the square $[0,i]^2$ (for $i=2,3,4$),
  by $H_g$ the hyperelliptic polygon $\conv \{ (0,0), (2g + 2,0), (0,2) \}$ (for $g = 7,8$), and
  by $E$ the exceptional polygon $\conv \{ (1,0), (3,1), (0,3) \}$ from the statement of Theorem~\ref{maintheorem}.
  Each entry
  corresponds to a sample of $10^4$ uniformly randomly chosen
  Laurent polynomials $f \in k[x^{\pm 1}, y^{\pm 1}]$ that are supported on $\square_2, \square_3, \dots$
  The table presents the proportion of $f$'s for which $\text{Jac}(C_f')$ has a non-trivial $k$-rational $2$-torsion point,
  among those $f$'s that are non-degenerate with
  respect to their Newton polygon $\Delta(f) = \square_2, \square_3, \dots$ The count was carried out
  using Magma~\cite{magma}, either by using the intrinsic function for computing the Hasse-Weil zeta function,
  or by spelling out the Hasse-Witt matrix \cite[Thm.~1.1]{stohrvoloch} and applying Manin's theorem \cite{manin}.\\

  \begin{center}
  \begin{tabular}{|r||c|c|c|c|c|c|c|}
    \hline
    $k \phantom{i}$ & $\begin{array}{c}\square_2 \\ \text{\footnotesize{$(g = 1)$}} \\ \end{array}$ & $\begin{array}{c}\square_3 \\ \text{\footnotesize{$(g = 4)$}} \\ \end{array}$ & $\begin{array}{c}\square_4 \\ \text{\footnotesize{$(g = 9)$}} \\ \end{array}$ & $\begin{array}{c} H_7 \\ \text{\footnotesize{$(g = 7)$}} \\ \end{array}$ & $\begin{array}{c} H_8 \\ \text{\footnotesize{$(g = 8)$}} \\ \end{array}$ & $\begin{array}{c} E \\ \text{\footnotesize{$(g = 3)$}} \\ \end{array}$ \\
    \hline \hline
    \rule{0pt}{3ex} $\mathbf{F}_2 \phantom{i}$ & $0/0$ & 0.370 & 0.958 & 0.995 & 0.670 & 0.143 \\
    $\mathbf{F}_4 \phantom{i}$ & 0.750 & 0.621 & 1.000 & 1.000 & 0.795 & 0.449 \\
    $\mathbf{F}_8 \phantom{i}$ & 0.884 & 0.654 & 1.000  & 1.000 & 0.852 & 0.591 \\
    $\mathbf{F}_{16} \phantom{i}$ & 0.940 & 0.697 & 1.000 & 1.000 & 0.872 & 0.661 \\
    $\mathbf{F}_{32} \phantom{i}$ & 0.968 & 0.704 & 1.000 & 1.000 & 0.877 & 0.696 \\
    $\mathbf{F}_{64} \phantom{i}$ & 0.986 & 0.716  & 1.000 & 1.000 & 0.880 & 0.694 \\
    $\mathbf{F}_{128} \phantom{i}$ & 0.992 & 0.703  & 1.000 & 1.000 & 0.889 & 0.708\\
    $\mathbf{F}_{256} \phantom{i}$ & 0.996 & 0.709  & 1.000 & 1.000 & 0.888 & 0.707\\
    $\begin{array}{r} \text{\footnotesize{asymptotic}} \\
                      \text{\footnotesize{prediction}} \\
     \end{array}$
    & $1$ & $\frac{32}{45} \approx 0.711$ & $1$ & $1$ & $\frac{8}{9} \approx 0.889$ & $\frac{5}{7} \approx 0.714$ \\
    \hline
  \end{tabular}
  \end{center}
  \vspace{0.4cm}
  Note that the conditions of Theorem~\ref{maintheorem} are
  satisfied for $\square_2$, $\square_4$ and $H_7$. So here we proved that the proportion converges to $1$.
  In the case of $H_8$, by the material in Section~\ref{hyperelliptic}
(see Corollary~\ref{cor:hyp5}) we know that the proportion converges
  to $\frac{8}{9}$. In the other two cases $\square_4$ and $E$ we have no clue, so our best guess is that these follow
  the $\text{GL}_g(\FF_2)$-model.

\subsubsection*{Growing polygon.} Let $k$ be a finite field of characteristic $2$. If $\Delta$ is a two-dimensional lattice polygon satisfying the conditions
  of Lemma~1, then the same holds for each odd Minkowski multiple $(2n + 1)\Delta$.
  It seems reasonable to assume that the proportion of $(2n+1)\Delta$-non-degenerate Laurent polynomials $f \in k[x^{\pm 1}, y^{\pm 1}]$, amongst all
  Laurent polynomials that are supported on $(2n+1)\Delta$, converges to
  a certain strictly positive constant.

  This is certainly true for the larger proportion of $(2n+1)\Delta$-\emph{toric} Laurent polynomials.
  Namely,
  using~\cite[Thm.~1.2]{poonen} one
  can show that this proportion converges to
  \[ Z_{\text{Tor}_k(\Delta) \setminus S}((\#k)^{-3})^{-1} \cdot Z_S((\#k)^{-1})^{-1} \]
  as $n \rightarrow \infty$; here $S$ denotes the (finite) set of singular points of $\text{Tor}_k(\Delta)$, and $Z$ stands for
  the Hasse-Weil Zeta function.
  It should be possible to prove a similar statement for non-degenerate Laurent polynomials
  by redoing the closed point sieve in the proof of~\cite[Thm.~1.2]{poonen}, but we did not work out
  the details of this.

  On the other hand, the number of solutions to~\eqref{congs} inside $(2n+1)\Delta \cap \ZZ^2$ tends to infinity. So
  the assumption would allow
  one to conclude:
  \[ \lim_{n \rightarrow \infty} \text{Prob} \left( \, \text{Jac}(C_f')(k)[2] \neq 0 \, \left| \, \text{$f \in k[x^{\pm 1}, y^{\pm 1}]$ is $(2n+1)\Delta$-non-degenerate} \, \right. \right) = 1. \]
  This is again deviating statistical behavior: in view of Cohen-Lenstra type heuristics, one naively expects a probability of about
  \[ 1 - \prod_{j = 1}^\infty (1 - 2^{-j}) \approx 0.71121;\]
  see~\cite{CEZB} for some additional comments.

  When applied to $(2n+1)\Sigma$-toric Laurent polynomials, where $\Sigma$ is
  the standard simplex, one recovers the claim made before~\cite[Thm.~4.2]{CEZB}.

\section{Connections with the rank of the Hasse-Witt matrix} \label{hassewitt}

Let us revisit the proof of Theorem~\ref{maintheorem} from the previous section. Our sufficient
condition that
\begin{equation} \label{condi}
 c_{i_0,j_0} \neq 0 \quad \text{for at least one solution $(i_0,j_0) \in \Delta \cap \mathbf{Z}^2$ to the system (\ref{congs})}
\end{equation}
(see right before Remark~\ref{hereagain}) seems
rather equation-specific.
However, it is easy to show that
automorphisms of $\text{Tor}_k(\Delta)$ cannot alter whether (\ref{condi}) is
satisfied or not. For instance, in the case of smooth plane projective curves of odd degree $d \geq 3$,
one verifies that if
\[ F(X,Y,Z) = \sum_{i + j \leq d} c_{i,j} X^i Y^j Z^{d - i - j}  \in k[X,Y,Z] \]
is such that $c_{i,j} = 0$ as soon as both $i$ and $j$ are odd, then
applying a linear change of variables does not affect this.
This suggests that something more fundamental is going on.
In Conjecture~\ref{hassewittconjecture} below we formulate a
guess for a geometric interpretation of
condition (\ref{condi}), involving the rank of the Hasse-Witt matrix,
and prove this guess in a number of special cases.
Our main references
on the Hasse-Witt matrix are \cite{manin,serre,stohrvoloch}.\footnote{More precisely, we follow~\cite[p.\,54]{stohrvoloch} and view the Hasse-Witt matrix as the Frobenius-conjugate transpose of the matrix of the \emph{Cartier-Manin operator} with respect to a basis of the space of regular differentials. The default definition of the Hasse-Witt matrix instead considers the \emph{Frobenius operator} acting on cohomology. Taking the Cartier-Manin detour was recently advised against by Achter and Howe~\cite{warning}, because it easily leads to misuses of Frobenius as a semilinear operator. But in this section we are only interested in the rank of the Hasse-Witt matrix, which is equal to the rank of the Cartier-Manin operator, so we stay out of the danger zone.} 

Here is a first fact:

\begin{lemma} \label{conjectureeasydirection}
  Let $k$ be a perfect field of characteristic $2$, let
  $\Delta$ be a two-dimensional lattice polygon satisfying the conditions of
  Lemma~\ref{combinatoriclemma}, and let
  $f = \sum_{(i,j) \in \ZZ^2} c_{i,j} x^i y^j \, \in k[x^{\pm 1}, y^{\pm 1}]$
  be a $\Delta$-non-degenerate (or $\Delta$-toric) Laurent polynomial.
  Let
  \begin{itemize}
   \item $g$ be the genus of $C_f'$, i.e.\ $g = \# (\Delta^{(1)} \cap \mathbf{Z}^2)$, and
   \item $\rho$ be the number of solutions $(i_0,j_0) \in \Delta \cap \mathbf{Z}^2$ to the system of congruences (\ref{congs}).
  \end{itemize}
  If $c_{i_0,j_0} = 0$ for every such solution, then the rank of the Hasse-Witt matrix of $C'_f$ is at most $g - \rho$.
\end{lemma}

\begin{proof}
 By \cite[Cor.~2.6~and~2.7]{CDV}
 we find that
 \begin{equation} \label{basisofdiffs}
  \left\{  x^iy^j \frac{dx}{xy \frac{\partial f}{\partial y}} \right \}_{(i,j) \in \Delta^{(1)} \cap \mathbf{Z}^2}
 \end{equation}
 is a basis for the space of regular differentials on $C'_f$. (If in the $\Delta$-toric case
 the denominator happens to vanish identically, one can
 replace $dx / (\partial f / \partial y)$ by $dy / (\partial f / \partial x)$ as explained in
 Remark~\ref{remarknondegversussmooth}.)
 Assume that $c_{i_0,j_0} = 0$ for each of the $\rho$ solutions
 $(i_0,j_0) \in \Delta \cap \mathbf{Z}^2$ to the system (\ref{congs}).
 Remark that these solutions are all contained in $\Delta^{(1)}$. One then verifies
 that the $\rho$ corresponding differentials $z_{i_0,j_0} dx$, where
 \[ z_{i_0,j_0} = \frac{ x^{i_0}y^{j_0} }{xy \frac{\partial f}{\partial y}}, \]
 satisfy $\partial z_{i_0,j_0} / \partial x = 0$.
 Following the construction from \cite[\S1]{stohrvoloch} we conclude that at least $\rho$ rows of
 the Hasse-Witt matrix with respect to the basis (\ref{basisofdiffs}) are zero.
\end{proof}

As an interesting corollary we obtain:

\begin{corollary} \label{ordinaryimplies2torsion}
  Let $k$ and $\Delta$ be as before and let $f$
  be a $\Delta$-non-degenerate (or $\Delta$-toric) Laurent polynomial over $k$.
  Assume moreover
  that $\Delta$ is not among the polygons excluded in the statement of Theorem~\ref{maintheorem}.
  If $C_f'$ is ordinary then it has a non-trivial $k$-rational $2$-torsion point on its Jacobian.
\end{corollary}

\begin{proof}
 In view of Lemma~\ref{solutionsinside}, the fact that $\Delta$ is not among the excluded polygons
 ensures that $\rho > 0$. A result by Serre \cite[Prop.\ 10]{serre} says that
 $C_f'$ is ordinary if and only if its Hasse-Witt matrix has rank $g$.
 So the previous lemma implies that if $C'_f$ is ordinary, then (\ref{condi}) is satisfied.
 The claim now follows from Theorem~\ref{maintheorem}.
\end{proof}

\begin{remark} \label{remarkritzenthaler}
The following alternative proof of Corollary~\ref{ordinaryimplies2torsion} was suggested to us by
Christophe Ritzenthaler. A result by St\"ohr and Voloch \cite[Cor.~3.2]{stohrvoloch}
states that the Hasse-Witt matrix has rank $g - h^0(C'_f, \Theta_\text{arith})$.
So if $C_f'$ is ordinary then $h^0(C'_f, \Theta_\text{arith}) = 0$, and in particular $\Theta_\text{arith}$
cannot be linearly equivalent to an effective divisor.
Now if $\Delta$ is not among the excluded polygons, then by Lemma~\ref{solutionsinside} there is at least one solution
$(i_0,j_0) \in \Delta \cap \mathbf{Z}^2$ to the system (\ref{congs}).
Fix such a solution and consider the corresponding translated polygon $(-i_0,-j_0) + \Delta$, as in the proof of
Lemma~\ref{combinatoriclemma}. We again find that all
$\langle \nu_\tau, p_\tau \rangle$'s are odd, but now because
$(0,0) \in (-i_0,-j_0) + \Delta$ we also find that they are strictly negative. In other words the
resulting half-canonical divisor $\Theta_\text{geom}$ is effective. Hence $\Theta_\text{geom}$ and
$\Theta_\text{arith}$ are non-equivalent. Their difference then yields
a non-trivial $k$-rational $2$-torsion point on $\text{Jac}(C_f')$.
\end{remark}

Our guess is that Lemma~\ref{conjectureeasydirection} admits the following converse. This would
give the desired geometric interpretation of condition (\ref{condi}).

\begin{conjecture} \label{hassewittconjecture}
  Let $k$ be a perfect field of characteristic $2$, let
  $\Delta$ be a two-dimensional lattice polygon satisfying the conditions from
  Lemma~\ref{combinatoriclemma}, and let $f$ be a $\Delta$-non-degenerate (or $\Delta$-toric) Laurent polynomial.
  Then the rank of the Hasse-Witt matrix of $C'_f$ is at least $g - \rho$,
  and the bound is attained if and only if $c_{i_0,j_0} = 0$ for every solution
  $(i_0,j_0) \in \Delta \cap \mathbf{Z}^2$ to the system of congruences (\ref{congs}).
\end{conjecture}

We can prove this conjecture in a number of special cases. Because the statements
seem interesting in their own right, we will each time reformulate (and sometimes refine) Conjecture~\ref{hassewittconjecture}
accordingly.

\begin{theorem}[Conjecture~\ref{hassewittconjecture} for smooth plane curves of odd degree] \label{conjsmoothplane}
 Let $k$ be a perfect field of characteristic $2$, let $d \geq 3$ be an odd integer and let
 $f = \sum_{i + j \leq d} c_{i,j} x^i y^j \in k[x,y]$
 define a smooth plane projective curve $C/k$ of degree $d$ and genus $g = (d-1)(d-2)/2$.
 Then the rank of the Hasse-Witt matrix of $C$ is bounded from below by
                \[ g - \frac{d^2 - 1}{8} = \frac{3}{8} (d-1)(d-3) \]
 Furthermore equality holds if and only if $c_{i,j} = 0$ as soon as $i$ and $j$ are odd.
\end{theorem}

\begin{proof}
  Recall from Remark~\ref{remarkritzenthaler} that St\"ohr and Voloch \cite[Cor.~3.2]{stohrvoloch} proved that
  the rank of the Hasse-Witt matrix is $g - h^0(C, \Theta_\text{arith})$.
  By the Brill-Noether theory of smooth plane curves \cite[Thm.~2.1]{hartshorne}
  we have
  \begin{equation} \label{brillnoetherupperbound}
   h^0(C, D) \leq \frac{ \frac{d-1}{2} \frac{d+1}{2}  }{2}
   = (d^2 - 1)/8
  \end{equation}
  for any divisor $D$ on $C$ of degree $g-1$. In particular this also holds
  for $D = \Theta_\text{arith}$, from which the lower bound follows.
  As for the last statement, by \cite[part 2b of Thm.~2.1]{hartshorne}
  the bound in (\ref{brillnoetherupperbound}) is attained
  if and only if $D$ is in the class of $\frac{d-3}{2}H$, i.e.\ if and only if
  $D \sim \Theta_\text{geom}$.
  But the proof of Theorem~\ref{maintheorem} (or of \cite[Thm.~4.2]{CEZB}) is
  precisely about showing that if $c_{i,j} \neq 0$ for some $i$ and $j$ that are both odd, then
  $\Theta_\text{arith} \not \sim \Theta_\text{geom}$. This yields the `only if' part, while
  the `if' part follows from Lemma~\ref{conjectureeasydirection}.
\end{proof}

\begin{theorem}[Conjecture~\ref{hassewittconjecture} for hyperelliptic curves of odd genus] \label{conjhyp}
 Let $k$ be a perfect field of characteristic $2$. Let $C$ be a hyperelliptic
 curve of odd genus $g \geq 3$, given in weighted projective form by
\begin{equation}\label{eq:hypmodel}
 C \colon \quad Y^2 + H_1(X,Z)Y = H_0(X,Z),
\end{equation}
where $H_1$
and $H_0$ in $k[X,Z]$ are homogeneous
of degrees $g+1$ and $2g+2$ respectively.
 Then the rank of the Hasse-Witt matrix of $C$ equals
 \[ g - \frac{1}{2} \deg \gcd \left( H_1, Z^{-1} \frac{\partial}{\partial X} H_1 \right). \]
  In particular, it
 is bounded from below by
 \[ g - \frac{g + 1}{2} = \frac{g - 1}{2}, \]
 where equality holds if and only if $\frac{\partial}{\partial X} H_1 = 0$.
\end{theorem}

\begin{proof}
 Write $H_1 = \sum_{i = 0}^{g+1} c_i X^iZ^{g+1 - i}$ and define
 \[ P(X,Z) = \sum_{i=0}^{(g+1)/2} c_{2i} X^iZ^{(g+1)/2 - i} \qquad \text{and}
 \qquad Q(X,Z) = \sum_{i=0}^{(g-1)/2} c_{2i+1} X^i Z^{(g-1)/2 - i}.
\]
Note that $H_1 = P^2 + XZQ^2$ and $\frac{\partial}{\partial X} H_1 = ZQ^2$.
Now the polynomial $f = y^2 + H_1(x,1)y + H_0(x,1)$ is $\Delta$-toric,
 where $\Delta = \text{conv} \{ (0,0), (2g+2,0), (0,2) \}$; here $C_f'$ is nothing else but $C$.
 An explicit computation shows that the Hasse-Witt matrix with respect to the basis (\ref{basisofdiffs})
 equals, up to a reordering of the rows, the Sylvester matrix of $P$ and $Q$.
 It is well-known
 that the corank of the Sylvester matrix of two polynomials equals the degree of their greatest
 common divisor, which in our case equals
 \[ \deg \gcd (P,Q) = \frac{1}{2} \deg \gcd (P^2, Q^2) = \frac{1}{2} \deg \gcd \left( H_1, Z^{-1} \frac{\partial}{\partial X} H_1 \right). \]
 The remaining claims follow immediately.
\end{proof}

\begin{remark}
  This indeed implies Conjecture~\ref{hassewittconjecture} for hyperelliptic curves
  of odd genus because $\frac{\partial}{\partial X} H_1 = 0$ if and only if all terms $c_{i,j}x^iy^j$ in
  $f = y^2 + H_1(x,1)y + H_0(x,1)$ with $i$ and $j$ odd are $0$.
\end{remark}

\begin{remark}
  The lower bound $(g-1)/2$ holds for arbitrary curves $C$ of genus $g$ (not necessarily odd) over fields
  of characteristic $2$, and it can be attained
  by hyperelliptic curves only. This follows from Clifford's theorem, as explained in \cite[Cor.\,3.2]{stohrvoloch}.
\end{remark}

\begin{theorem}[Conjecture~\ref{hassewittconjecture} for the exceptional polygons] \label{conjexceptional}
 Let $k$ be a perfect field of characteristic $2$, let $\Delta$ be one of
 the polygons
 \begin{center}
\psset{unit=0.3cm}
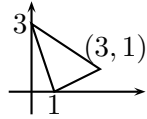
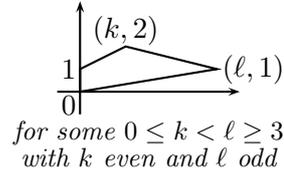
\begin{pspicture}(-1,-3.2)(5,4)
\pspolygon[fillstyle=solid,linecolor=black](1,0)(0,3)(3,1)
\psline{->}(-1,0)(5,0) \psline{->}(0,-1)(0,4)
\rput(1,-0.55){\small $1$} \rput(-0.5,3){\small $3$}
\rput(3.7,1.9){\small $(3,1)$}
\end{pspicture}
\qquad \qquad
\begin{pspicture}(0,-3.2)(4,4)
\rput(2,1.5){or}
\end{pspicture}
\qquad \qquad
\begin{pspicture}(-4,-3.2)(6,4)
\pspolygon[fillstyle=solid,linecolor=black](0,0)(6,1)(2,2)(0,1)
\psline{->}(-1,0)(7,0) \psline{->}(0,-1)(0,4)
\rput(-0.5,1){\small $1$}
\rput(-0.5,-0.6){\small $0$}
\rput(3,-1.8){\footnotesize for some $0 \leq k < \ell \geq 3 $}
\rput(3,-2.9){\footnotesize with $k$ even and $\ell$ odd}
\rput(7.6,1){\small $(\ell,1)$}
\rput(2,2.7){\small $(k,2)$}
\end{pspicture}
\end{center}
 that were excluded in the statement of Theorem~\ref{maintheorem},
 and let $f \in k[x^{\pm 1}, y^{\pm 1}]$ be $\Delta$-non-degenerate (or $\Delta$-toric).
 Then
 the rank of the Hasse-Witt matrix of $C'_f$ is equal to $g = \# (\Delta^{(1)} \cap \mathbf{Z}^2)$.
 In particular $C_f'$ is ordinary.
\end{theorem}

\begin{proof}
 The polygon on the left corresponds to smooth plane quartics of the form
 \[ c_{1,0} XZ^3 + c_{1,1}XYZ^2 + c_{2,1}X^2YZ + c_{3,1}X^3Y + c_{1,2}XY^2Z + c_{0,3}Y^3Z. \]
 The Hasse-Witt matrices of smooth plane quartics are explicitly described at the end of \cite[\S3]{stohrvoloch}. In our case
 this gives
 \[ \begin{pmatrix} c_{1,1} & c_{3,1} & 0 \\ 0 & c_{2,1} & c_{0,3} \\ c_{1,0} & 0 & c_{1,2} \\ \end{pmatrix} \]
 with determinant $c_{1,1}c_{2,1}c_{1,2} + c_{1,0}c_{3,1}c_{0,3}$.
 With the aid of a computer algebra package
 one can verify that this determinant is non-zero (using that the curve is smooth).

 As for the polygons on the right, we have that $f = c x^ky^2 + h_1(x)y + c'$
 for non-zero $c,c' \in k$ and a degree $\ell = g+1$ polynomial $h_1(x) \in k[x]$. Substituting
 $y \leftarrow yx^{-k}$ and multiplying the equation by $c^{-1}x^k$ puts our curve in the Weierstrass form
 \[ y^2 + c^{-1}h_1(x)y + c^{-1}c'x^k. \]
 Using that $k$ is even one sees that $h_1(x)$ is square-free (otherwise there would be an affine
 singularity). The result then follows from the previous theorem.
\end{proof}

A fun corollary is the following geometric sufficient condition for ordinarity. Remark that similar
conditions have been described before (such as the existence of $7$ bitangent lines, which
is actually sufficient and necessary; see \cite[\S3]{stohrvoloch}).

\begin{corollary}
  Let $C$ be a smooth plane quartic curve over a field $k$ of characteristic $2$ admitting three non-colinear inflection points,
  such that the corresponding tangent lines are
  precisely the lines through two of these points.
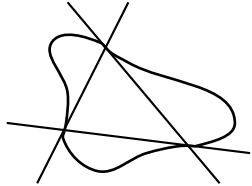
\begin{center}
\psset{unit=0.4cm}
  \begin{pspicture}(-1,0)(7,7)
\psline(0,0)(3,6)
\psline(-1,2)(7,1)
\psline(6,0)(1,6)
\psccurve[linewidth=1pt](4.75,1.24)(5,1.24)(5.25,1.24)(6.5,2)(3,4)(2.4,4.35)(2.25,4.5)(2.1,4.65)(0.5,4.7)(1,3)(0.95,2)(0.9,1.75)(0.85,1.5)(2,0.4)(3.6,1)
\end{pspicture}
\end{center}
  Then $C$ is ordinary.
\end{corollary}

\begin{proof}
A projective transformation positions the three inflection points at $(0:0:1)$, $(0:1:0)$ and $(1:0:0)$. One verifies
that the dehomogenization of the corresponding defining polynomial is $\Delta$-non-degenerate, where $\Delta$ is the left-most
polygon in the statement of the previous corollary.
\end{proof}

\begin{remark}
  Theorems~\ref{conjsmoothplane},~\ref{conjhyp} and~\ref{conjexceptional} provide several characteristic $2$ examples
  of families of curves whose Hasse-Witt matrices have constant rank. This (partly) addresses Question~2 of \cite[\S3.7]{farnellpries}.
\end{remark}

\section{Hyperelliptic curves} \label{hyperelliptic}

Let $C$ be a hyperelliptic curve
of genus $g \geq 2$
over a perfect field~$k$.
Then $C$ has a smooth weighted projective plane model of the form (\ref{eq:hypmodel}).
The Newton polygon of (the defining polynomial of) the corresponding
affine model $y^2 + H_1(x,1)y - H_0(x,1)=0$
is contained in
a triangle
with vertices $(0,0)$, $(2g+2,0)$ and $(0,2)$,
and is generically equal to this triangle.
In particular, Theorem~\ref{maintheorem} implies
that if the characteristic of $k$ is $2$
and $C$ is sufficiently general of odd genus, then its Jacobian
has a non-trivial $k$-rational $2$-torsion point.
By Corollary~\ref{ordinaryimplies2torsion}
we can replace `sufficiently general' by `ordinary'.

The purpose of this stand-alone section is to give alternative proofs of these facts
(Corollaries~\ref{cor:hyp2} and~\ref{cor:hyp5}),
using an explicit description of the $2$-torsion subgroup of $\text{Jac}(C)$.

\begin{theorem}\label{thm:hyp}
Let $C/k$ be a hyperelliptic curve over a perfect
field $k$ of characteristic~$2$
given by a smooth model~\eqref{eq:hypmodel}.
%
%
The Jacobian of $C$ has \emph{no}
rational point of order $2$
if and only if $H_1(X,Z)$
is a power of an irreducible odd-degree
polynomial in $k[X,Z]$.
\end{theorem}
\begin{corollary}\label{cor:hyp1}
Let $C/k$ be a hyperelliptic curve of odd $2$-rank
over a perfect field $k$ of characteristic~$2$.
Then the Jacobian of $C$ has a $k$-rational point
of order $2$.
\end{corollary}
\begin{corollary}\label{cor:hyp2}
Let $C/k$ be an ordinary hyperelliptic curve of odd genus
over a perfect field $k$ of characteristic~$2$.
Then the Jacobian of $C$ has a $k$-rational point
of order $2$.
\end{corollary}
\begin{corollary}\label{cor:hyp4}
Let $C/k$ be a hyperelliptic curve of genus $2^m-1$ over a perfect field
$k$ of characteristic~$2$, for some integer~$m \geq 2$.
If the Jacobian of $C$ has no $k$-rational point of order~$2$,
then it has $2$-rank zero, but it is not supersingular.
\end{corollary}

\noindent Finally, for integers $g, r \geq 1$, let $c_{g,r}$ be the proportion of
equations (\ref{eq:hypmodel}) over $\FF_{2^r}$
that define a curve of genus $g$ whose
Jacobian has at least one rational point of order~$2$.

\begin{corollary}\label{cor:hyp5}
The limit $\lim_{r\rightarrow\infty} c_{g,r}$ exists and
we have
\[ \lim_{r\rightarrow\infty} c_{g,r} = \begin{cases}
   1 & \text{if $g$ is odd,} \\
   g / (g + 1) & \text{if $g$ is even}.
\end{cases} \]
\end{corollary}

\begin{proof}[Proof of Theorem~\ref{thm:hyp}]
All we need to do is describe the two-torsion of the Jacobian
$\text{Jac}(C)$ of~$C$. Since we were not able to find a ready-to-use statement in the literature,
we give a stand-alone treatment, even though what follows
is undoubtedly known to several experts in the field; for instance, it
is implicitly contained in~\cite{EP,prieszhu}. Let $\overline{k}$ be an algebraic closure
of $k$. Note that $C$ has a unique point $Q_{(a:b)} = (a : \sqrt{F(a,b)} : b)\in C(\overline{k})$
for every root $(a:b)\in\mathbf{P}^1_{\overline{k}}$ of~$H_1=H_1(X,Z)$.
This gives $n$ points, where $n  \in \{1, \dots, g+1 \}$ is the number
of distinct roots of~$H_1$.
Let $D$ be the divisor of zeroes of a vertical line,
so $D$ is effective of degree~$2$.
All such divisors $D$ are linearly equivalent,
and are linearly equivalent to $2Q_{(a:b)}$ for each~$(a:b)$.
In particular, if we let
$$A = \ker \left( \xymatrix{\displaystyle \bigoplus_{(a:b)} (\ZZ/2\ZZ) \ar^{\sum}[r] & (\ZZ/2\ZZ)} \right),$$
then we have a homomorphism
\begin{align*}
A & \quad\longrightarrow \quad \text{Jac}(C)(\overline{k})[2]\\
(c_{(a:b)}\ \mbox{mod}\ 2)_{(a:b)} & \quad\longmapsto \quad(\sum_{(a:b)} c_{(a:b)} Q_{(a:b)} ) - (\frac{1}{2}\sum_{(a:b)} c_{(a:b)})D.
\end{align*}
In fact, this map is an isomorphism. Indeed, it is injective
because if the divisor of a function is invariant under the hyperelliptic involution, then so
is the function itself, i.e.\ it is contained in $\overline{k}(x)$. But at the points $Q_{(a:b)}$
such functions can only admit poles or zeroes having an even order.
Surjectivity follows from the fact that $\text{Jac}(C)(\overline{k})[2]$ is generated by divisors that are supported on
the Weierstrass locus of $C$. This can be seen using Cantor's algorithm
\cite[Appendix.\S6-7]{Koblitz}, for the application of which one needs to
transform the curve to a so-called imaginary model; this is always possible over $\overline{k}$.
Alternatively,
surjectivity follows from the injectivity and the fact that $\# \text{Jac}(C)(\overline{k})[2] = 2^{n-1}$ by
\cite[Thm.~1.3]{EP}.

Then in particular, the rational $2$-torsion subgroup
$\text{Jac}(C)(k)[2]$
is isomorphic to the subgroup of elements of~$A$
that are invariant under $\mathrm{Gal}(\overline{k}/k)$,
that is, to
\begin{equation*}
A_k = \ker
\left(
\bigoplus_{P\mid H_1} (\ZZ/2\ZZ) \rightarrow (\ZZ/2\ZZ) : (c_P)_P  \mapsto \sum_{P} c_P\deg(P)
\right)
\end{equation*}
where the sum is taken over the irreducible factors $P$ of~$H_1$.

The only way for~$A_k$ to be trivial is for~$H_1$
to be the power of an irreducible factor~$P$
of odd degree.
\end{proof}

\begin{proof}[Proof of Corollary~\ref{cor:hyp1}]
Let $n$ be the degree of the radical $R$ of~$H_1$.
The $2$-rank of $C$ equals $n-1$
(as in the proof of Theorem~\ref{thm:hyp}; see e.g.~\cite[Thm.~1.3]{EP}).
So if the $2$-rank is odd, then~$R$ has even degree, which
implies that $H_1$ is not a power of an irreducible odd-degree polynomial.
In particular, Theorem~\ref{thm:hyp}
implies that $C$ has a non-trivial $k$-rational $2$-torsion point.
\end{proof}
\begin{proof}[Proof of Corollary~\ref{cor:hyp2}]
This is a special case of Corollary~\ref{cor:hyp1}
since in characteristic~$2$, the $2$-rank of an ordinary abelian variety equals
its dimension.
\end{proof}
\begin{proof}[Proof of Corollary~\ref{cor:hyp4}]
If there is no rational point of order~$2$, then $H_1$ is a power of
a polynomial of odd degree dividing $\deg\,H_1=g+1=2^m$.
In other words, it is a power
of a linear polynomial and hence the $2$ rank of $C$ is zero.
There are no supersingular hyperelliptic
curves of genus $2^m-1$ in characteristic~$2$ by~\cite[Thm.~1.2]{SZ}.
\end{proof}
\begin{proof}[Proof of Corollary~\ref{cor:hyp5}]
As $r \rightarrow \infty$, the proportion of equations (\ref{eq:hypmodel})
for which $H_1$ is not separable becomes negligible.
By Theorem~\ref{thm:hyp} it therefore
suffices to prove the corresponding limit for the proportion
of degree $g + 1$ polynomials that are \emph{not} irreducible of odd degree.
If $g$ is odd then this proportion is clearly $1$. If $g$ is even
then this
is the same as the proportion of reducible polynomials of degree $g+1$,
which converges to $1 - (g+1)^{-1}$.
\end{proof}

\begin{remark} \label{isomorphism}
  In Corollary~\ref{cor:hyp5}, instead of working with the proportion of equations (\ref{eq:hypmodel}),
  we can work with the corresponding proportion of $\FF_{2^r}$-isomorphism classes of
  hyperelliptic curves of genus $g$.  This is because
  the subset of equations~(\ref{eq:hypmodel}) that define a hyperelliptic curve of genus $g$ whose
  only non-trivial geometric automorphism
  is the hyperelliptic involution (inside the affine space of all
  equations of this form) is non-empty~\cite{bjornaut}, open, and defined over $\FF_2$ (being invariant under
  the $\text{Gal}(\overline{\FF}_2, \FF_2)$-action). See also~\cite{zhu}.
\end{remark}

We finish by identifying the $2$-torsion point
from the proof of Theorem~\ref{maintheorem} in the hyperelliptic case
with one of the $2$-torsion points from the proof of Theorem~\ref{thm:hyp}.
The former proof
provides
$\Theta_{\mathrm{arith}}$ and $\Theta_{\mathrm{geom}}$
with $2\Theta_{\mathrm{arith}} \sim 2\Theta_{\mathrm{geom}}$,
hence the class of
$T = \Theta_{\mathrm{arith}} - \Theta_{\mathrm{geom}}$
is two-torsion.
We have $2\Theta_{\mathrm{arith}}=\text{div} \, dx$.
To compute $2\Theta_{\mathrm{geom}}$, we need to take
an appropriate model as in the proof of
Lemma~\ref{combinatoriclemma}.
The bivariate polynomial $y^2 + H_1(x,1)y + H_0(x,1)$
gives an affine model of our hyperelliptic curve~$C$,
and if $g$ is odd, then
the system from Lemma~\ref{combinatoriclemma} admits the solution $(1,1)$.
By the proof of that lemma, we should then look at the
toric model $C'_f$ where $$f=x^{-1}(y + H_1(x,1) + y^{-1}H_0(x,1)).$$
Then $\Theta_{\mathrm{geom}}$
is given by $2\Theta_{\mathrm{geom}} = \text{div} \frac{1}{xy \frac{\partial f}{\partial y}}dx$,
so
we compute
\[
\frac{\partial f}{\partial y} = x^{-1} (1+y^{-2}H_0(x,1))
=x^{-1}y^{-1}\ H_1(x,1).
\]
We find
\[
T = \Theta_{\mathrm{arith}} - \Theta_{\mathrm{geom}}
 =\frac{1}{2}\ \text{div}\, xy \frac{\partial f}{\partial y}
 =\frac{1}{2} \text{div}\, H_1(x,1),
\]
where $ \text{div} \, H_1(x,1)$ is twice the sum of all points
$P_{(a:b)}$ as $(a:b)$ ranges over the roots of~$H_1(X,Z)$ in $\mathbf{P}^1_{\overline{k}}$ (with multiplicity),
minus $(g+1)$ times the divisor~$D$ of degree~$2$ at infinity.
This is the $2$-torsion point from the proof of Theorem~\ref{thm:hyp}
corresponding to the element $(1,1,\ldots,1)\in A_k$.

\section*{Acknowledgements}
We sincerely thank Christophe Ritzenthaler, Arne Smeets and the anonymous referees for several helpful comments. The
first author was supported financially by FWO-Vlaanderen.

\vspace{0.5cm}

\noindent \verb"wouter.castryck@gmail.com".\\
\noindent Departement Wiskunde, KU Leuven, Celestijnenlaan 200B, 3001 Leuven, Belgium.\\

\noindent \verb"marco.streng@gmail.com".\\
\noindent Mathematisch Instituut, Universiteit Leiden, Postbus 9512, 2300 RA Leiden, The Netherlands.\\

\noindent \verb"d.testa@warwick.ac.uk".\\
\noindent Mathematics Institute, University of Warwick, Coventry CV4 7AL, United Kingdom.\\


\begin{thebibliography}{99}
  \bibitem[AH]{warning} J.\,Achter, E.\,Howe, \emph{Hasse-Witt and Cartier-Manin matrices: a warning and a request}, Proc.\ of `Arithmetic, Geometry, Cryptography, and Coding Theory', Contemporary Mathematics~\textbf{722}, pp.\ 1-18 (2019)
  \bibitem[BCP]{magma} W.\,Bosma, J.\,Cannon, C.\,Playoust, \emph{The Magma algebra system. I. The user language},
  J.\ Symbolic Comput.\ \textbf{24}, pp.\ 235–265 (1997)
  \bibitem[CEZB]{CEZB} B.\,Cais, J.\,Ellenberg, D.\,Zureick-Brown, \emph{Random Dieudonn\'e modules, random $p$-divisible groups, and random curves over finite fields}, to appear in J.\ Math.\ Inst.\ Jussieu
  \bibitem[Cas]{movingout} W.\,Castryck, \emph{Moving out the edges of a lattice polygon}, Discrete and Computational Geometry \textbf{47}(3), pp.\ 496-518 (2012)
  \bibitem[CC]{linearpencils} W.\,Castryck, F.\,Cools, \emph{Linear pencils encoded in the Newton polygon}, preprint
  \bibitem[CDV]{CDV} W.\,Castryck, J.\,Denef, F.\,Vercauteren, \emph{Computing zeta functions of nondegenerate curves}, Int.\ Math.\ Res.\ Pap.\ \textbf{2006}, pp.\ 1-57 (2006)
  \bibitem[CFHS]{CFHS} W.\,Castryck, A.\,Folsom, H.\,Hubrechts, A.V.\,Sutherland, \emph{The probability that the number of points on
  the Jacobian of a genus $2$ curve is prime}, Proc.\ London Math.\ Soc.\ \textbf{104}(6), pp.\ 1235-1270 (2012)
  \bibitem[CV]{CastryckVoight} W.\,Castryck, J.\,Voight, \emph{On nondegeneracy of curves}, Algebra \& Number Theory \textbf{3}(3), pp.\ 255-281 (2009)
  \bibitem[CLS]{coxlittleschenck} D.\,Cox, J.\,Little, H.\,Schenck, \emph{Toric varieties}, Graduate Studies in Mathematics \textbf{124} (2011)
  \bibitem[DV1]{DV1} J.\,Denef, F.\,Vercauteren, \emph{Computing zeta functions of hyperelliptic curves over finite fields of characteristic $2$}, Proc.\
  of `Advances in Cryptology -- CRYPTO 2002', Lect.\ Not.\ Comp.\ Sc.\ \textbf{2442}, pp.\ 308-323 (2002)
  \bibitem[DV2]{DV2} J.\,Denef, F.\,Vercauteren, \emph{Computing zeta functions of $C_{a,b}$ curves using Monsky-Washnitzer cohomology}, Fin. Fields App.\ \textbf{12}(1), pp.\ 78-102 (2006)
  \bibitem[EP]{EP} A.\,Elkin, R.\,Pries, \emph{Ekedahl-Oort strata of hyperelliptic curves in characteristic $2$}, Algebra \& Number Theory \textbf{7}(3), pp.\ 507-532 (2013)
  \bibitem[FP]{farnellpries} S.\,Farnell, R.\, Pries, \emph{Families of Artin-Schreier curves with Cartier-Manin matrix of constant rank}, Linear Algebra and its Applications \textbf{439}(7), pp.\ 2158-2166 (2013)
  \bibitem[HS]{HaaseSchicho} C.\,Haase, J.\,Schicho, \emph{Lattice polygons and the number $2i+7$}, American Mathematical Monthly \textbf{116}(2), pp.\ 151-165 (2009)
  \bibitem[Har]{hartshorne} R.\,Hartshorne, \emph{Generalized divisors on Gorenstein curves and a theorem of Noether}, J. Math. Kyoto Univ. \textbf{26}(3), pp.\ 375-386 (1986)
  \bibitem[KS]{KatzSarnak} N.\,Katz, P.\,Sarnak, \emph{Random matrices, Frobenius eigenvalues, and monodromy},
American Mathematical Society (1999)
  \bibitem[Kob]{Koblitz} N.\,Koblitz, \emph{Algebraic aspects of cryptography}, Algorithms and Computation in Mathematics \textbf{3}, Springer (1999)
   \bibitem[Koe]{Koelman} R.\,Koelman, \emph{The number of moduli of families of curves on toric surfaces},
 Ph.D. thesis, Katholieke Universiteit Nijmegen (1991)
  \bibitem[Man]{manin} Y.\,Manin, \emph{The Hasse-Witt matrix of an algebraic curve}, AMS Translations, Series 2 \textbf{45}, pp.\ 245–264 (1965)
(originally in Izv.\ Akad.\ Nauk SSSR Ser.\ Mat.\ \textbf{25}, pp.\ 153-172 (1961))
  \bibitem[Mum]{mumford} D.\,Mumford, \emph{Theta characteristics of an algebraic curve}, Ann.\ Sci.\ de l'\'E.N.S.\ \textbf{4}(2), pp.\ 181-192 (1971)
  \bibitem[Poo1]{bjornaut} B.\,Poonen, \emph{Varieties without extra automorphisms. II. Hyperelliptic curves},
Math.\ Res.\ Lett.\ \textbf{7} (1), pp.\ 77-82 (2000)
  \bibitem[Poo2]{poonen} B.\,Poonen, \emph{Bertini theorems over finite fields}, Ann.\ Math.\ \textbf{160}, pp.\ 1099-1127 (2004)
  \bibitem[PZ]{prieszhu} R.\,Pries, H.\,Zhu, \emph{The $p$-rank stratification of Artin-Schreier curves}, Annales de l'Institut Fourier \textbf{62}(2), pp.\ 707-726 (2012)
  \bibitem[Ser]{serre} J.-P.\,Serre, \emph{Sur la topologie des vari\'et\'es alg\'ebriques en caract\'eristique $p$}, Oeuvres (collected papers) \textbf{1}, Springer, pp.\ 544-568 (1986)
  \bibitem[SZ]{SZ} J.\,Scholten, H.\,Zhu, \emph{Hyperelliptic curves in characteristic $2$}, Int.\ Math.\ Res.\ Not.\ \textbf{2002}(17), pp.\ 905-917 (2002)
  \bibitem[SV]{stohrvoloch} K.-O.\,St\"ohr, J.\,F.\,Voloch, \emph{A formula for the Cartier operator on plane algebraic curves}, Journal f\"ur die reine und angewandte Mathematik \textbf{377}, pp.\ 49-64 (1987)
  \bibitem[Zhu]{zhu} H.\,Zhu, \emph{Hyperelliptic curves over $\FF_2$ of every $2$-rank without extra automorphisms},
Proc.\ Amer.\ Math.\ Soc.\ \textbf{134}(2), 323-331 (2006)
\end{thebibliography}
\end{document}